\documentclass[12pt]{article}
\usepackage{amssymb,amsfonts,amsthm}
\usepackage{amsmath,amsthm,amssymb}
\usepackage{amscd}
\usepackage{amsfonts}
\usepackage{color}
\newcommand{\vv}{{\mathbb V}}

\newcommand{\con}{\operatorname{con}}

\newcommand{\mul}{\operatorname{mul}}
\newcommand{\simp}{\operatorname{sim}}
\newcommand{\ini}{\operatorname{ini}}

\newcommand{\var}{\operatorname{var}}

\newtheorem{theorem}{Theorem}[section]
\newtheorem{ex}[theorem]{Example}
\newtheorem{cor}[theorem]{Corollary}
\newtheorem*{sufcon}{Sufficient Condition}
\newtheorem{fact}[theorem]{Fact}
\newtheorem{lemma}[theorem]{Lemma}

\newtheorem{prop}[theorem]{Proposition}
\newtheorem{obs}[theorem]{Observation}

\newtheorem*{claim}{Claim}

\begin{document}

\date{}

\title{Limit varieties generated by finite non-J-trivial aperiodic monoids}
\author{Olga  B. Sapir}

\maketitle

\begin{abstract}   Jackson and Lee proved that  certain six-element monoid  generates a hereditarily finitely based variety $\mathbb E^1$  whose  lattice of subvarieties  contains an infinite ascending chain.
 We identify syntactic monoids which generate  finitely generated subvarieties of $\mathbb E^1$ and show that one of these finite monoids together with certain seven-element monoid generates a  new limit  variety.

\end{abstract}

\section{Introduction}
\label{sec: intr}

A variety of algebras is called \textit{finitely based} (abbreviated to FB) if it has a finite basis of its identities, otherwise, the variety is said to be \textit{non-finitely based} (abbreviated to NFB).
A variety is called  \textit{limit} if it is NFB but all its proper subvarieties are FB.

The first two explicit examples of limit monoid varieties $\mathbb L$ and $\mathbb M$ were discovered by Jackson~\cite{Jackson-05} in 2005.
In~\cite{Lee-09}, Lee proved the uniqueness of the limit varieties $\mathbb L$ and $\mathbb M$ in the class of varieties of finitely generated aperiodic monoids with central idempotents.
In~\cite{Lee-12}, Lee generalized the result of~\cite{Lee-09} and established that  $\mathbb L$ and $\mathbb M$ are the only limit varieties within the class of varieties of aperiodic monoids with central idempotents.
In 2013, Zhang found a NBF variety of monoids that contains neither $\mathbb L$ nor $\mathbb M$~\cite{Zhang-13} and, therefore, she proved that there exists a limit variety of monoids that differs from $\mathbb L$ and $\mathbb M$. In~\cite{Zhang-Luo}, Zhang and Luo pointed out an explicit example of such variety.

The  following semigroup  was introduced by its multiplication table and shown to be FB in \cite[Section~ 19]{Lee-Zhang}.
Its presentation  was recently suggested by Edmond W. H. Lee:
\[
A=\langle a,b,c\mid a^2=a,\,b^2=b,\,ab=ca=0,\,ac=cb=c\rangle=\{a,b,c,ba,bc,0\}.
\]
If $S$ is a semigroup, then the monoid obtained by adjoining a new identity element to $S$ is denoted by $S^1$ and the variety of monoids generated by $S^1$ is denoted by $\mathbb S^1$.
If $\mathbb V$ is a monoid variety, then  $\overline{\mathbb V}$ denotes the variety \textit{dual to} $\mathbb V$, i.e., the variety consisting of monoids anti-isomorphic to monoids from $\mathbb V$.
The variety $\mathbb A^1 \vee\overline{\mathbb A^1}$ is the third example of limit variety of monoids~\cite{Zhang-Luo} mentioned in the previous paragraph. Figure~\ref{pic} below contains the lattice
of subvarieties of $\mathbb A^1$ duplicated from Figure~1 in \cite{Zhang-Luo}.

The next pair of limit varieties  $\mathbb J$ and $\mathbb {\overline J}$  was discovered by Gusev \cite{Gusev-SF}.
In~\cite{Gusev-JAA}, he proves that   $\mathbb L$, $\mathbb  M$,
$\mathbb J$ and $\mathbb {\overline J}$  are the only limit varieties of aperiodic monoids with commuting idempotents.
In~\cite{GS},  Gusev and the author present the last pair  $\mathbb K$ and $\mathbb {\overline K}$ of limit varieties generated by finite $J$-trivial monoids 
 and show that there are exactly seven limit varieties of $J$-trivial monoids.

Let $E$ be the semigroup given by presentation:
\[
E = \langle a, b, c \mid a^2 = ab = 0, ba = ca = a, b^2 = bc = b, c^2= cb = c \rangle=\{a,b,c,ac,0\}.
\]
Monoid $E^1$ was first investigated by Lee and Li \cite[Section~14]{Lee-Li}, where it was shown to be finitely based by $\{xtx \approx xtx^2 \approx x^2tx,\, xy^2x \approx x^2y^2\}$.

 Elements of a countably infinite alphabet $\mathfrak A$ are called {\em letters} and elements of the free monoid $\mathfrak A^*$  are called {\em words}. We use $1$ to denote the empty word, which is the identity element of $\mathfrak A^*$.
Words unlike letters are written in bold.

Denote:
\[{\bf u}_0 = a^2b^2,  {\bf v}_0 = b^2a^2;\]
\begin{equation} \label{unvn} {\bf u}_{k+1} = at_{k+1} {\bf u}_k,  {\bf v}_{k+1} = at_{k+1} {\bf v}_k, k=0, 2, 4, \dots;\end{equation}
\[{\bf u}_{k+1} = bt_{k+1} {\bf u}_k,  {\bf v}_{k+1} = bt_{k+1} {\bf v}_k, k=1, 3, 5, \dots.\]
 For example:
\[ {\bf u}_1= at_1a^2b^2,  {\bf v}_1 = at_1b^2a^2;\]
\[ {\bf u}_2= bt_2at_1a^2b^2,  {\bf v}_2 =bt_2 at_1b^2a^2;\]
\[ {\bf u}_3= a t_3 bt_2at_1a^2b^2,  {\bf v}_3 =a t_3 bt_2 at_1b^2a^2;\]
\[ \dots \hskip.4in  \dots\]
For each $n \ge 1$,  the identity $\sigma_n$ introduced in  \cite{Jackson-Lee}
is equivalent modulo $xtx \approx xtx^2 \approx x^2tx$ to $\mathbf {u}_{n} \approx \mathbf {v}_{n}$.
We use $\var \Sigma$ to denote the variety defined by a set of identities $\Sigma$.
According to Proposition 5.6 in~\cite{Jackson-Lee}, the lattice
 of subvarieties of $\mathbb E^1$ contains
an infinite ascending chain $\mathbb E^1\{\sigma_1\} \subset \mathbb E^1\{\sigma_2\} \subset \dots$, where for each $n\ge 1$,
\[\mathbb E^1 \{\sigma_n\} = \var \{xtx \approx xtx^2 \approx x^2tx,\, xy^2x \approx x^2y^2, \mathbf {u}_{n} \approx \mathbf {v}_{n}\}.\]
A copy of this lattice from Figure~4 in \cite{Jackson-Lee} is shown on Figure~\ref{pic} below.
The variety $\mathbb E^1$ is the first example of a finitely generated monoid variety whose lattice of subvarieties  is countably infinite.
The second example of a variety with this property  is contained in \cite{GLZ}.

In Sect.~\ref{sec: sc}, we present  a Sufficient Condition under  which the variety $\mathbb A^1 \vee {\mathbb E}^1\{\sigma_2\}$
  is NFB.  Using Sufficient Condition in \cite{GS} under which a monoid is FB, we show that  every proper subvariety of $\mathbb A^1 \vee {\mathbb E}^1\{\sigma_2\}$ and of $\overline{\mathbb A}^1 \vee \overline{{\mathbb E}^1\{\sigma_2\}}$ is FB.
Hence  $\mathbb A^1 \vee {\mathbb E}^1\{\sigma_2\}$ and $\overline{\mathbb A}^1 \vee \overline{{\mathbb E}^1 \{\sigma_2\}}$ are new limit varieties of monoids.

Let $A_0$ be the semigroup given by presentation:
\[A_0 = \langle  a, b \mid a^2=a, b^2=b, ab = 0 \rangle = \{a,b, ba, 0\}.\]
The monoid $A_0^1$ was  shown to be FB in \cite{Edmunds-77}.

It turns out \cite{GS-23},  that  ${\mathbb E}^1\{\sigma_2\} \vee \overline{{\mathbb E}^1\{\sigma_2\}} \vee \mathbb A_0^1$ is also a limit variety.

The finite monoids which generate limit varieties $\mathbb L$ and $\mathbb M$ were introduced by Jackson in 2005 by using the so-called Dilworth-Perkins construction, which assigns a monoid $M(W)$ to a set of words $W$.
More precisely, $\mathbb L = M(\{abtbsa, atbsba\})$  and $\mathbb M = M(\{atbasb\})$ (see ~\cite{Jackson-05}).
In \cite{Sapir-18}, we generalized Dilworth-Perkins construction into $M_\tau(W)$ construction for monoids and $S_\tau(W)$ construction for semigroups, where $\tau$ is a congruence
on $\mathfrak A^*$ (resp. $\mathfrak A^+$ ).
Surprisingly,  each of the ten limit varieties mentioned above can be generated by  monoids of the form $M_\tau(W)$ where $W = [{\bf u}]_\tau$ is an equivalence class of a word ${\bf u} \in \{abtbsa, atbsba, atbasb, atb^2a, ab^2ta, ab\}$ and $\tau$ is some easy-to-define congruence on the free monoid (see Sect.~\ref{sec: 14}). 

Another well known construction assigns a syntactic monoid  $M_{synt}(W)$  or syntactic semigroup $S_{synt}(W)$ to a set of words $W$.
Proposition~2.1 in \cite{Jackson-03} about syntactic algebras implies that every monoid (semigroup with zero)  is equationally equivalent to a syntactic monoid (syntactic semigroup).

The usefulness and simplicity of Dilworth-Perkins construction and the universality of syntactic algebras
motivated us to find a connection between monoids of the form $M_\tau(W)$ (resp. semigroups of the form $S_\tau(W)$) and the syntactic monoids $M_{synt}(W)$ (resp. syntactic semigroups $S_{synt}(W)$). Theorem~\ref{T: SM} gives us such a connection and is used in Sect.~\ref{sec: MW} and Sect.~\ref{sec: 14}   to identify syntactic algebras which generate some $0$-simple semigroups, finitely generated subvarieties of $\mathbb E^1$, subvarieties of $\mathbb A^1$, and the ten limit varieties of monoids mentioned above. By computing syntactic algebras using the generalized Dilworth-Perkins construction
we avoid any computations of the syntactic congruence, which tend to be very cumbersome.

 Recently, Gusev et al. \cite{GLZ} found two more pairs of limit varieties of monoids. The subvariety lattices of these varieties are much more complex than of the varieties mentioned above. In particular, it is shown in \cite{GLZ} that one of these varieties contains infinitely many infinite ascending chains of subvarieties. It is the first example of a finitely generated variety with this property.
It can be deduced from \cite{GLZ} that, in contrast to the ten limit varieties of monoids discussed above, the $M_\tau(W)$ formulas for the varieties in \cite{GLZ} are too bulky to be useful.

\section{Congruences $\tau_1$, $\gamma$ and $\beta$ on the free monoid $\mathfrak A^\ast$} \label{sec: cong}

We say that a set of words $W \subseteq \mathfrak A^*$ is {\em stable with respect to a semigroup variety $\vv$} if  ${\bf v} \in W$ whenever ${\bf u} \in W$ and $\vv$ satisfies ${\bf u} \approx {\bf v}$. Recall that a word ${\bf u} \in  \mathfrak A^*$ is an {\em isoterm} \cite{Perkins-69} for $\vv$ if the set $\{ {\bf u} \}$ is stable with respect to $\vv$.

If $\tau$ is an equivalence relation on the free monoid $\mathfrak A^*$ and $\vv$ is  a semigroup variety,
then a word ${\bf u} \in \mathfrak A^*$ is said to be  a {\em $\tau$-term} for $\vv$ if ${\bf u} \tau {\bf v}$ whenever $\vv$ satisfies ${\bf u} \approx {\bf v}$. Notice that if  $W \subseteq \mathfrak A^*$ forms a single  $\tau$-class, then  $W$ is stable with respect to $\vv$ if and only if every word in  ${\bf u} \in W$ is a $\tau$-term for $\vv$.

 Let $\tau_1$ denote the congruence on the free monoid ${\mathfrak A}^*$ induced by the relations $a=a^2$ for each $a \in{\mathfrak A}$.

A letter is called {\em simple} ({\em multiple}) in a word $\bf u$ if it occurs in $\bf u$ once (at least
twice).   The set of all  letters in $\bf u$ is denoted by $\con({\bf u})$. Notice that  $\con({\bf u}) = \simp({\bf u}) \cup \mul({\bf u})$
where $\simp({\bf u})$ is the set of all simple letters in $\bf u$ and  $\mul({\bf u})$ is the set of all multiple letters in $\bf u$.

Let $\gamma$ be the fully invariant congruence of $\var \{xy \approx yx, x^2 \approx x^3\}$.
 It is well-known and can be easily verified that this variety is generated by the 3-element  monoid  $\langle a, 1 \mid a^2 = 0\rangle$ and that for every ${\bf u}, {\bf v} \in \mathfrak A^\ast$, we have
${\bf u}\mathrel{\gamma}{\bf v}$ if and only if   $\simp({\bf u}) = \simp({\bf v})$ and  $\mul({\bf u}) = \mul({\bf v})$.

Given $W \subseteq \mathfrak A^*$ we use $W^\le$ to denote the set of all subwords of words in $W$.

\begin{lemma} \cite[Corollary~3.5]{Sapir-20+} \label{L: gammasub}  Suppose that $W \subseteq \mathfrak A^*$ forms a single
 $(\tau_1 \wedge \gamma)$-class (resp. $\tau_1$-class) of  $\mathfrak A^*$.  If  $W$ is  stable with respect to a monoid variety $\vv$   then every word in $W^\le$ is a  $(\tau_1 \wedge \gamma)$-term (resp. $\tau_1$-term) for $\vv$.
\end{lemma}

 A \textit{block} of a word $\mathbf u$ is a maximal subword of $\mathbf u$ that does not contain any letters simple in $\mathbf u$.

 We use $_{i{\bf u}}x$ to refer to the $i$th from the left occurrence of $x$ in a word ${\bf u}$.
We use $_{\ell {\bf u}}x$ to refer to the last occurrence of $x$ in ${\bf u}$.  If $x$ is simple in $\bf u$ then we use $_{{\bf u}}x$ to
denote the only occurrence of $x$ in $\bf u$.
If  the $i$th occurrence of $x$ precedes  the $j$th occurrence of $y$ in $\bf u$, we write $({_{i{\bf u}}x}) <_{\bf u} ({_{j{\bf u}}y})$.

Let $\beta$ be the fully invariant congruence  of  $\mathbb E^1 =  \var \{xtx \approx xtx^2 \approx x^2tx, xy^2x \approx x^2 y^2\}$. It follows from \cite[Section 14.3]{Lee-Li} that for every ${\bf u}, {\bf v} \in \mathfrak A^\ast$, we have
${\bf u}\mathrel{\beta}{\bf v}$ if and only if

\begin{itemize}

\item (i) the identity ${\bf u} \approx {\bf v}$ is of the form:

\[ \mathbf a_0 \prod_{i=1}^m (t_i\mathbf a_i) \approx \mathbf b_0 \prod_{i=1}^m (t_i\mathbf b_i),\]
where $\simp(\mathbf u)=\simp(\mathbf v) = \{t_1, \dots, t_m\}$ for some $m \ge 0$ and \\
$\mul(\mathbf u) = \con(\mathbf a_0  \dots \mathbf a_m)  = \con(\mathbf b_0  \dots \mathbf b_m) =\mul(\mathbf v)$;

\item (ii) for each $i=0,1,\dots, m$ we have $\con({\bf a}_i) = \con({\bf b}_i)$;

\item (iii) for each  $i=0,1,\dots, m$ and for each $x \ne y \in \con({\bf a}_i) = \con({\bf b}_i)$, we have
$(_{1{{\bf a}_i}} x) <_{{\bf a}_i} {(_{1{{\bf a}_i}}y)} \Leftrightarrow (_{1{{\bf b}_i}} x) <_{{\bf b}_i} {(_{1{{\bf b}_i}} y)}$.

\end{itemize}

For example, $(x^2y^3xtyx^2)\mathrel{\beta}(xy ty^3xy)$.
 Let  $\overline{\beta}$ denote the congruence  dual to $\beta$.
We use $\sim_{Q^1}$ to  denote the equivalence relation on $\mathfrak A^*$ given by the conditions (i) and (ii) only.
The proof of Proposition~4.3 in \cite{Lee-Li} implies that $\sim_{Q^1}$ is the
fully invariant congruence of $\mathbb Q^1$, where $Q^1$ is  the monoid obtained  by adjoining an identity element to the following  semigroup:
\[Q = \langle e, b, c \mid  e^2 = e, eb = b, ce = c, ec = be = cb = 0 \rangle.\]
The semigroup $Q$ was introduced and shown to be FB  in \cite[Section 6.5]{JA}.
Proposition~4.3 in \cite{Lee-Li} establishes that
$\mathbb Q^1 = \var \{ xtx \approx xtx^2 \approx x^2tx, x^2y^2\approx y^2x^2 \}$.

A word $\bf u$ is called {\em block-simple} if every block of $\bf u$ involves at most one letter. For example, the word $x^2t_1x^3t_2y^5t_3z^2t_4t_5y$  is block-simple. It is easy to see that if ${\bf u} \in \mathfrak A^*$ is block-simple then every word in
 $[{\bf u}]_{\tau_1 \wedge \gamma} = [{\bf u}]_\beta = [{\bf u}]_{\overline{\beta}}$  is block-simple.

We use regular expressions to describe sets of words, in particular the contents of congruence classes. For example, $[atb^2a]_\beta$ consists of all words of the form $a^nt {\bf s}$, where $n \ge 1$ and
${\bf s} \in \{a,b\}^+$ starts with $b$, contains $b$ at least twice and contains $a$. Using regular expression, we write $[atb^2a]_\beta = a^+ t ba^+b\{a,b\}^* \vee  a^+ t bb^+a\{a,b\}^*$.

\begin{fact} \label{F: xtsx} For a monoid variety $\vv$, the following are equivalent:

(i)  the set $[xtsx]_{\tau_1 \wedge \gamma}  = [xtsx]_{\beta}  =  x^+tsx^+$ is  stable with respect to  $\vv$;

(ii) $\vv$ contains $\mathbb Q^1$;

(iii) every word is $(\sim_{Q^1})$-term for $\vv$;

(iv) every block-simple  word is $(\tau_1 \wedge \gamma)$-term for $\vv$;

(v) every block-simple  word is $\beta$-term for $\vv$.

\end{fact}

\begin{proof}  The equivalence of (i) and (ii) follows immediately from Proposition~2.3 and Theorem~4.3(i) in \cite{Sapir-20+}.
Parts (ii) and (iii) are also equivalent because $\sim_{Q^1}$ is the
fully invariant congruence of $\mathbb Q^1$.
The equivalence of (iii) and (iv) are easily verified.
Parts (iv) and (v) are equivalent because a block-simple word is a $(\tau_1 \wedge \gamma)$-term for $\vv$ if and
only if it is a $\beta$-term for $\vv$.
\end{proof}

A word $\bf w$ is called {\em almost-block-simple} if for each $x \ne y \in \mul({\bf w})$
at most one  block in $\bf w$ involves both $x$ and $y$. For example, the word $xt_1yt_2x^2yzt_3zt_4yp^3t_5xp$ is almost-block-simple.
Notice that every block-simple word is almost-block-simple and every word in \eqref{unvn} is almost-block-simple.

\begin{obs} \label{O: two classes} Let $\bf w$ be an almost block-simple word where $\bf b$ is the only block which involves letters $a \ne b$.
If ${\bf b} \in  a^+b \{a,b\}^*$ then $[{\bf w}]_{\sim_{Q^1}} =[{\bf w}]_\beta \cup [{\bf w'}]_\beta$ where ${\bf w'}$ is obtained from $\bf w$ by replacing $\bf b$ by $b^2a^2$.
\end{obs}

\begin{lemma} \label{L: un} Let ${\bf u}_1,  {\bf u}_2, {\bf u}_3, \dots; {\bf v}_1,  {\bf v}_2, {\bf v}_3, \dots$ be the words defined recursively in \eqref{unvn}.
 If for some $n \ge 1$ either $[{\bf u}_n]_\beta$ or $[{\bf v}_n]_\beta$ is stable with respect to a monoid variety $\vv$ then every almost-block-simple word with
at most $n$ simple letters is $\beta$-term for $\vv$.

\end{lemma}

\begin{proof} First, we establish the following.

\begin{claim} Every block-simple  word is  $\beta$-term for $\vv$.

\end{claim}

\begin{proof}  We verify this claim only for ${\bf v}_1= at b^2a^2$, because the argument for  ${\bf u}_n$  and ${\bf v}_n$ with $n \ge 1$ is similar but bulkier.
So, we are given that the set $[atb^2a^2]_\beta = [atb^2a]_\beta$ is stable with respect to $\vv$. 
First, notice that $t$ must be an isoterm for $\vv$, because otherwise,
$\vv$ satisfies  $xty^2x \approx  x t^k y^2x$   for some $k \ge 2$, which contradicts the fact that $xty^2x$ is a $\beta$-term for $\vv$.
Since an assumption that $\vv$ is commutative also easily leads to a contradiction, the word $xy$ must be an isoterm for $\vv$.

If  the set $x^+tx^+$ is not stable with respect to $\vv$, then $\vv \models x^nt x^m \approx x^kt$ or  $\vv \models x^nt x^m \approx tx^k$  for some $n, m \ge 1$, $k \ge 2$. The first identity implies 
$x^nty^2x^m \approx x^kty^2$ and the second implies $x^nty^2x^m \approx ty^2x^k$. Since the left side of each of these identities is in $[xty^2x]_\beta$ but the right side is not, the set  $x^+tx^+$ must be stable with respect to $\vv$.

If $x^+tsx^+$ is not stable with respect to $\vv$, then $\vv \models x^n t s x^m \approx x^ktx^q sx^p$ for some $n, m, k, q, p \ge 1$.
This identity implies  $x^n t y^2 x^m \approx x^ktx^q y^2x^p$, which contradicts the fact that $x^nty^2x^m$ is a $\beta$-term for $\vv$.
We conclude that the set $x^+tsx^+$ is stable with respect to $\vv$. Consequently, 
every block-simple  word is  a $\beta$-term for $\vv$ by Fact~\ref{F: xtsx}.\end{proof}

Next we verify the following.

\begin{claim} For each $i=0, \dots, n$ both  $[{\bf u}_i]_\beta$ and $[{\bf v}_i]_\beta$ are stable with respect to  $\vv$.

\end{claim}

 \begin{proof} Since  every block-simple  word is  $\beta$-term for $\vv$,  Fact~\ref{F: xtsx}(iii) implies that 
for each $i=0, \dots, n$, the set
\[[{\bf u}_i]_{\sim_{Q^1}} = [{\bf v}_i]_{\sim_{Q^1}} \stackrel {Observation~\ref{O: two classes}}{=}
[{\bf u}_i]_\beta \cup [{\bf v}_i]_\beta\]
is stable with respect to  $\vv$.  Since $[{\bf u}_n]_\beta$ or $[{\bf v}_n]_\beta$ is stable with respect to  $\vv$,
both $[{\bf u}_n]_\beta$ and  $[{\bf v}_n]_\beta$ are stable with respect to  $\vv$.  If
either  $[{\bf u}_{n-1}]_\beta$ or  $[{\bf v}_{n-1}]_\beta$ is not stable with respect to  $\vv$, then
$\vv \models {\bf u} \approx {\bf v}$ such that ${\bf u} \in [{\bf u}_{n-1}]_\beta$ but ${\bf v} \in [{\bf v}_{n-1}]_\beta$.
This identity implies $at_n{\bf u} \approx at_n{\bf v}$, where  $at_n{\bf u} \in [{\bf u}_{n}]_\beta$ but $at_n{\bf v} \in [{\bf v}_{n}]_\beta$.
To avoid the contradiction, both $[{\bf u}_{n-1}]_\beta$ and  $[{\bf v}_{n-1}]_\beta$ must be stable with respect to  $\vv$.
And so on, until we show that both $[{\bf u}_{0}]_\beta$ and  $[{\bf v}_{0}]_\beta$ are stable with respect to  $\vv$.\end{proof}

In view of Fact~\ref{F: xtsx} and the first claim every identity of $\vv$ holds on $\mathbb Q^1$.
Now let $\bf w$ be an almost-block-simple word with at most $n$ simple letters.
To obtain a contradiction, assume that $\bf w$ is not a $\beta$-term for $\vv$. Then $\vv$ satisfies
an identity ${\bf w} \approx {\bf w'}$ such that  ${\bf w} \sim_{Q^1} {\bf w'}$, and for some block $\bf a$ in $\bf w$ and $x \ne y \in \mul({\bf a})$ we have 
$(_{1{\bf a}} x) <_{\bf a} {(_{1{\bf a}}y)}$ but  $(_{1{\bf b}} y) <_{\bf b} {(_{1{\bf b}}x)}$, where $\bf b$ is the block in $\bf w'$ which
corresponds to the block $\bf a$ in $\bf w$.  Let $\{t_1, \dots, t_k\} \subseteq \simp({\bf w})$ be the (possibly empty) set of simple
letters which appear in $\bf w$ on the left of block $\bf a$.
Since $\bf w$ is almost-block-simple, for some $\mathfrak T \subseteq \{t_1, \dots, t_k\}$, the variety  $\vv$ satisfies ${\bf w}(x,y, \mathfrak T) \approx {\bf w'}(x,y, \mathfrak T)$ such that
modulo renaming letters ${\bf w}(x,y, \mathfrak T) \in [{\bf u}_m]_\beta$  but  ${\bf w'}(x,y, \mathfrak T) \in [{\bf v}_m]_\beta$, where $m$ 
is the number of letters in $\mathfrak T$. Since $0 \le m \le k \le n$, this contradicts the second claim. To avoid the contratiction we conclude that every almost-block-simple word must be $\beta$-term for $\vv$.
\end{proof}

\section{Sufficient condition under which a monoid is NFB} \label{sec: sc}

\begin{fact}  \cite[Fact~2.1]{Sapir-15N} \label{F: nfb} Suppose that for infinitely many $n$, a semigroup
 variety $\vv$  satisfies an identity ${\bf U}_n \approx {\bf V}_n$ in at least $n$ letters such that  ${\bf U}_n$ has some property $P_n$ but  ${\bf V}_n$ does not. Suppose that for every word $\bf U$ such that $\vv \models {\bf U} \approx {\bf U}_n$ and ${\bf U}$  has property $P_n$,
  the word $\Theta({\bf v})$ also has property $P_n$ for
every substitution
$\Theta: \mathfrak A \rightarrow \mathfrak A^+$ and every identity ${\bf u} \approx {\bf v}$ in less than, say, $n/2$ letters such that $\Theta({\bf u}) = {\bf U}$.   Then $\vv$ is NFB.

\end{fact}

If  $\mathbf U=\Theta(\mathbf u)$ for some endomorphism $\Theta$ of $\mathfrak A^+$ and $_{i{\bf U}}x$ is an occurrence of a letter $x$ in $\bf U$ then  $\Theta^{-1}_{\bf u}({_{i{\bf U}}x})$ denotes  an occurrence  ${_{j{\bf u}}z}$ of a letter $z$ in  $\bf u$ such that $\Theta({_{j{\bf U}}z})$ regarded as a subword of $\bf U$ contains $_{i{\bf U}}x$.

\begin{sufcon}
Let $\vv$ be a monoid variety that  satisfies the identity
\begin{equation} \label{long identity1} {\bf U}_n = x y_1^2y_2^2\cdots y^2_{n-1} y_n^2x\approx x y_1^2 x y_2^2 \cdots y^2_{n-1} x y^2_n x = {\bf V}_n\end{equation}
for any $n\ge1$. If the sets  $[ab^2ta]_{\tau_1 \wedge \gamma} = a^+bb^+ta^+$ and $[at b^2 a]_{\beta}$ are stable with respect to $\vv$ then $\vv$ is NFB.

\end{sufcon}

\begin{proof}
Consider the following property of a word ${\bf U}$ with $\con({\bf U})  = \{x, y_1, \dots,y_n\}$:

(P): There is no  $x$  in $\bf U$  between  the first occurrence of $y_1$ and the first occurrence  of $y_n$.

Notice that ${\mathbf U}_n$ satisfies property (P) but  ${\bf V}_n$ does not.

Let $\bf U$ be such that $\vv \models {\bf U}_n \approx {\bf U}$. Since $[a^2b^2]_{\tau_1 \wedge \gamma} = aa^+bb^+$ is stable with respect to $\vv$   by Lemma~\ref{L: gammasub},  we have:
\begin{equation}
\label{letters in u}
({_{1{\bf U}}x})   <_{\bf U} ({_{\ell{\bf U}}y_1})  <_{\bf U} ({_{1{\bf U}}y_2}) <_{\bf U} ({_{\ell{\bf U}}y_2}) <_{\bf U} \dots <_{\bf U} ({_{\ell{\bf U}}y_{n-1}}) <_{\bf U} ({_{1{\bf U}}y_n}) <_{\bf U} ({_{\ell{\bf U}}x}).
\end{equation}
Let ${\bf u} \approx {\bf v}$ be an identity of $\vv$ in less than $n/2$ letters and let
 $\Theta: \mathfrak A \rightarrow \mathfrak A^+$  be a substitution such that $\Theta({\bf u}) = {\bf U}$.
In view of \eqref{letters in u}, the following holds:

(*) If $\Theta(t)$ contains both $y_i$ and $y_j$ for some $1\le i <j \le n$ then letter $t$ is simple in $\bf u$.

Suppose that $\bf U$ has Property (P). Then in view of \eqref{letters in u}, every subword $\bf A$ of $\bf U$
has the following property:

(**) if $\bf A$ contains $xy_i$ then $i=1$; if  $\bf A$ contains $y_j x$ then $j=n$.

Let us verify that ${\bf V} = \Theta({\bf v})$ also has Property (P).
To obtain a contradiction, assume that there is  an occurrence of $x$ in $\bf V$ such that
\[({_{1{\bf V}}y_1}) <_{\bf V} ({_{k{\bf V}}x}) <_{\bf V} ({_{1{\bf V}}y_n}).\]
Consider two cases.

{\bf Case 1:}  There is  an occurrence of $x$ in $\bf V$ such that
\begin{equation}\label{case2}
({_{1{\bf V}}y_{n/2}}) <_{\bf V} ({_{k{\bf V}}x}) <_{\bf V} ({_{1{\bf V}}y_{n}}).
\end{equation}

Since $\bf u$ has less than $n/2$ letters, for some $t \in \con({\bf u})$ the word $\Theta(t)$ contains both $y_i$ and $y_j$ for some $1 \le i <j \le n/2$. In view of (*), the letter $t$ is simple in $\bf u$. Since $t$ is an isoterm for $\vv$ by Lemma~\ref{L: gammasub}, the letter $t$ is simple in $\bf v$ as well. $\Theta^{-1}_{\bf v}({_{k{\bf V}}x}) = {_{p{\bf v}}z}$ is an occurrence of some letter $z$ in $\bf v$ such that $\Theta(z)$ contains $x$.
Since the empty word $1$ is an isoterm for $\vv$ by Lemma~\ref{L: gammasub}, the letter $z$ occurs in $\bf u$ as well.

In view of Fact~2.6 in \cite{Sapir-15N}, $\Theta^{-1}_{\bf u}({_{1{\bf U}}y_n}) = {_{1{\bf u}}y}$ and
 $\Theta^{-1}_{\bf v}({_{1{\bf V}}y_n}) = {_{1{\bf v}}y'}$  for some $y,y' \in \con({\bf u})=\con({\bf v})$.
If $y \ne y'$ then $(_{1{\bf u}} y) <_{\bf u} {(_{1{\bf u}}y')}$ but $(_{1{\bf v}} y') <_{\bf v} {(_{1{\bf v}}y)}$.
 This is impossible, because ${\bf u}(y,y')$ is $\beta$-term for $\vv$  by by Lemma~\ref{L: un}.
 Thus $y=y'$.

If $t=z$ (resp. $y=z$) then in view of \eqref{case2},
$\Theta(t)=\Theta(z)$ (resp. $\Theta(y)=\Theta(z)$) contains either $xy_i$ for some $1 <i \le n$ or $y_j x$ for some $1 \le j <n$.
Since both $\Theta(t)$  and $\Theta(y)$ are subwords of $\bf U$, this is impossible by Property (**).
Therefore, $t \ne z$ and $y \ne z$.

Since $\bf U$ has Property (P), no $z$ occurs between $t$ and ${_{1{\bf u}}y}$ in $\bf u$. Hence  ${\bf u}(z,y,t) \in  z^*t z^*$ if $t=y$ and   ${\bf u}(z,y,t) \in  z^*t y \{y, z\}^*$ if $t \ne y$.
On the other hand, in view of \eqref{case2}, we have  $({_{{\bf v}}t}) <_{\bf v} ({_{p{\bf v}}z}) <_{\bf v} ({_{1{\bf v}}y})$. If $t=y$ this is impossible, because $t$ is simple in $\bf v$.
If $t \ne y$ then  ${\bf v}(z,y,t)
 \in  z^*t z \{y, z\}^*$.   This is  impossible, because  ${\bf u}(z,y,t)$ is $\beta$-term for $\vv$ by Lemma~\ref{L: un}.

{\bf Case 2:}  There is  an occurrence of $x$ in $\bf V$ such that
\begin{equation}\label{case1}
({_{1{\bf V}}y_1}) <_{\bf V} ({_{k{\bf V}}x}) <_{\bf V} ({_{1{\bf V}}y_{n/2}}).
\end{equation}

Since $\bf u$ has less than $n/2$ letters, for some $t \in \con({\bf u})$ the word $\Theta(t)$ contains both $y_i$ and $y_j$ for some $n/2 \le i <j \le n$. In view of (*), the letter $t$ is simple in $\bf u$. Since $t$ is an isoterm for $\vv$, $t$ is simple in $\bf v$ as well.
$\Theta^{-1}_{\bf v}({_{k{\bf V}}x}) = {_{p{\bf v}}z}$ is an occurrence of some letter $z$ in $\bf v$ such that $\Theta(z)$ contains $x$. Since the empty word $1$ is an isoterm for $\vv$ by Lemma~\ref{L: gammasub}, the letter $z$ occurs in $\bf u$ as well. 

In view of Fact~2.6 in \cite{Sapir-15N},
$\Theta^{-1}_{\bf u}({_{1{\bf U}}y_1}) = {_{1{\bf u}}y}$ and
 $\Theta^{-1}_{\bf v}({_{1{\bf V}}y_1}) = {_{1{\bf v}}y'}$  for some $y,y' \in \con({\bf u})=\con({\bf v})$.
If $y \ne y'$ then $(_{1{\bf u}} y) <_{\bf u} {(_{1{\bf u}}y')}$ but $(_{1{\bf v}} y') <_{\bf v} {(_{1{\bf v}}y)}$.
 This is impossible, because ${\bf u}(y,y')$ is $\beta$-term for $\vv$  by Lemma~\ref{L: un}.
 Thus $y=y'$.

If $t=z$ (resp. $y=z$) then in view of \eqref{case1},
$\Theta(t)=\Theta(z)$ (resp. $\Theta(y)=\Theta(z)$) contains either $xy_i$ for some $1 <i \le n$ or $y_j x$ for some $1 \le j <n$.
Since both $\Theta(t)$  and $\Theta(y)$ are subwords of $\bf U$, this is impossible by Property (**).
Therefore, $t \ne z$ and $y \ne z$.

{\bf Subcase 2.1:}  $({_{1{\bf u}}y})   <_{\bf u}  ({_{1{\bf u}}z})$.

In this case,   $({_{1{\bf v}}y})   <_{\bf v}  ({_{1{\bf v}}z})$ because  ${\bf u}(z,y)$ is $\beta$-term for $\vv$  by Lemma~\ref{L: un}. Since $\bf U$ has Property (P), we have $({_{{\bf u}}t}) <_{\bf u} ({_{1{\bf u}}z})$. But in view of \eqref{case1}, we have
 $({_{{1\bf v}}z}) <_{\bf v} ({_{{\bf v}}t})$. This is  impossible, because  ${\bf u}(z,t)$ is $\beta$-term for $\vv$  by  Lemma~\ref{L: un}.

{\bf Subcase 2.2:}  $({_{1{\bf u}}z})   <_{\bf u}  ({_{1{\bf u}}y})$.

In this case, $({_{1{\bf v}}z})   <_{\bf v}  ({_{1{\bf v}}y})$  because  ${\bf u}(z,y)$ is $\beta$-term for $\vv$  by Lemma~\ref{L: un}.
Since  $\bf U$ has Property (P), we have ${\bf u}(z,y,t) \in  z^+tz^*$ if $t=y$ and  ${\bf u}(z,y,t) \in  z^+y^+tz^*$ if $t \ne y$.
But in view of \eqref{case1}, the word
$\bf v$ contains an occurrence of $z$ between ${_{1{\bf v}}y}$ and $t$. If $t=y$ this is impossible, because $t$ is simple in $\bf v$.
If $t \ne y$, this contradicts the fact that
$[zy^2tz]_{\tau \wedge \gamma}$ and $[z^2y^2t]_{\tau \wedge \gamma}$ are stable with respect to $\vv$.

Since we obtain a contradiction in every case
we conclude that $\bf V$ must also satisfy Property (P). Therefore, the variety $\vv$ is NFB by  Fact~\ref{F: nfb}.
\end{proof}

\begin{cor} \label{C: AE} Every monoid variety $\vv$ that contains $\mathbb A^1 \vee  \mathbb E^1 \{\sigma_2\}$ and is contained in $\mathbb A^1 \vee \mathbb E^1$ is NFB.

\end{cor}

\begin{proof}The fact that 
$\mathbb A^1$  satisfies \eqref{long identity1}  is verified in \cite{Zhang-Luo}.
The variety $\mathbb E^1$ satisfies \eqref{long identity1} because ${\bf U}_n\mathrel{\beta}{\bf V}_n$ where $\beta$ is the 
fully invariant congruence of $\mathbb E^1$. 
Therefore,  $\mathbb A^1 \vee \mathbb E^1$ satisfies \eqref{long identity1}.

Theorem~4.3(iii)  in \cite{Sapir-20+} 
 implies  that   $[ab^2ta]_{\tau_1 \wedge \gamma} = a^+bb^+ta^+$ is stable with respect to $\mathbb A^1$ and can be used to recheck
that $\mathbb A^1$  satisfies \eqref{long identity1}.
The fact that  $[at b^2 a]_{\beta}$  is stable with respect to $\mathbb E \{\sigma_2\}$ is, in essence, verified in the proof
of Lemma~5.7 in  \cite{Jackson-Lee}.
 Hence $\vv$ is NFB by  the Sufficient Condition.
\end{proof}


\section{New pair of limit varieties of aperiodic monoids} \label{sec: lv}

Given a congruence $\tau$ on the free monoid  $\mathfrak A^*$, we use $\circ$ to denote the binary operation on the quotient monoid  $\mathfrak A^*/\tau$. We refer to the elements of $\mathfrak A^*/\tau$ as $\tau$-classes.
The subword relation $\le$  on $\mathfrak A^*$ can be naturally extended to $\tau$-classes as follows.
Given two $\tau$-classes ${\mathtt u}, {\mathtt v} \in \mathfrak A^*/\tau$  we write ${\mathtt v} \le_\tau {\mathtt u}$ if ${\mathtt u} = {\mathtt p}\circ_\tau {\mathtt v}\circ_\tau {\mathtt s}$ for some   ${\mathtt p}, {\mathtt s} \in \mathfrak A^*/\tau$.

Let $\tau$ be  a congruence  on the free monoid $\mathfrak A^*$  and  $W \subseteq \mathfrak A^*$ be a union of $\tau$-classes.
 If  $\mathfrak A^* = W^\le$ then  we define  $M_\tau(W) = M_\tau(\mathtt W) =  \mathfrak A^*/ \tau$, where $\mathtt W$ is the set of  all $\tau$-classes formed by words in $W$. If $\mathfrak A^* \setminus W^\le$ is not empty then it is a union of  $\tau$-classes containing all words which are not subwords of any word in $W$.
In this case we define $M_\tau(W) = M_\tau({\mathtt W})$ as the Rees quotient of $\mathfrak A^*/ \tau$  over the ideal $(\mathfrak A^*/ \tau) \setminus  {\mathtt W}^{\le_\tau}$, where
${\mathtt W}^{\le_\tau}$  is the closure of $\mathtt W$ in quasi-order  $\le_\tau$.

Here is the connection between monoids of the form $M_\tau(W)$ and
$\tau$-terms for monoid varieties.

\begin{prop} \cite[Proposition~2.3]{Sapir-20+}  \label{P: MW}  Let $\tau$ be a congruence on the free monoid $\mathfrak A^*$ such that the empty word $1$  forms a singleton $\tau$-class.
Let $W \subseteq \mathfrak A^*$ be a set of words which is a union of  $\tau$-classes.  Let  ${\mathtt W}  \subseteq \mathfrak A^*/ \tau$ denote the set of all $\tau$-classes contained in $W$. Then for every monoid variety $\vv$  the following are equivalent:

(i)  $\vv$ contains $M_\tau(W) = M_\tau({\mathtt W})$;

(ii)  every word in $W^\le$ is  $\tau$-term for $\vv$;

(iii)  every $\tau$-class in ${\mathtt W}^{\le_\tau}$ is stable with respect to $\vv$.

\end{prop}

If $\tau$ is the trivial congruence on $\mathfrak A^*$ then we simply write
$M(W)$ instead of $M_\tau(W)$.
Proposition \ref{P: MW} generalizes  Lemma 3.3 in \cite{Jackson-05} which gives a connection between monoids of the form $M(W)$ and isoterms for monoid varieties.

For each set of words $W$ we use  $\mathbb M_\tau(W)$ to denote
the monoid variety generated by $M_\tau(W)$.

\begin{ex} \label{E: E2}

 For each $n \ge 1$,

(i) $\mathbb E^1 \{\sigma_{n+1}\} = \mathbb M_\beta([{\bf v}_n]_{\beta})$, in particular, $\mathbb E^1 \{\sigma_2\} = \mathbb M_\beta([at b^2 a]_{\beta})$;

(ii)  a monoid variety $\vv$ contains $E^1 \{\sigma_{n+1}\}$ if and only if  $\beta$-class $[{\bf v}_n]_{\beta}$ is stable
with respect to $\vv$.

\end{ex}

\begin{proof} (i) The argument used in the proof of Lemma~5.7 in \cite{Jackson-Lee}, in essence,
establishes that the set $[{\bf v}_n]_{\beta}$  is stable with respect to $\mathbb E^1 \{\sigma_{n+1}\}$ for each $n \ge 1$.
Since  every word in $([{\bf v}_n]_{\beta})^\le$ is $\beta$-term for  $\mathbb E^1 \{\sigma_{n+1}\}$ by Lemma~\ref{L: un}, we have $\mathbb E^1 \{\sigma_{n+1}\} \supseteq \mathbb M_\beta([{\bf v}_n]_{\beta})$ by  Proposition \ref{P: MW}.

According to Proposition~5.6   in  \cite{Jackson-Lee}, for each $n \ge 1$, the variety $\mathbb E^1\{\sigma_n\}$ is a
unique  maximal subvariety of  $\mathbb E^1 \{\sigma_{n+1}\}$. Since $\mathbb E^1 \{\sigma_n\}$ satisfies  $ {\bf u}_n \approx {\bf v}_n$,
the word  ${\bf v}_n$ is not $\beta$-term for  $\mathbb E^1 \{\sigma_n\}$. In view of Proposition \ref{P: MW}, the variety  $\mathbb E^1 \{\sigma_n\}$ does not contain  $M_\beta([{\bf v}_n]_{\beta})$. Therefore,  $\mathbb E^1 \{\sigma_{n+1}\} = \mathbb M_\beta([{\bf v}_n]_{\beta})$.

(ii)  If   $[{\bf v}_n]_{\beta}$ is stable
with respect to $\vv$, then every word in $([{\bf v}_n]_{\beta})^\le$ is $\beta$-term for $\vv$ by Lemma~\ref{L: un}.
Hence $\vv$ contains  $M_\beta ([{\bf v}_n]_{\beta})$ by Proposition \ref{P: MW}.
Consequently, $\vv$ contains $\mathbb E^1 \{\sigma_{n+1}\}  = \mathbb M_\beta([{\bf v}_n]_{\beta})$ by Part (i).  Conversely, if $\vv$ contains $\mathbb E^1 \{\sigma_{n+1}\}  = \mathbb M_\beta([{\bf v}_n]_{\beta})$ then  $[{\bf v}_n]_{\beta}$ is stable with respect to $\vv$ by  Proposition \ref{P: MW}.
\end{proof}

\begin{fact} \label{F: FB} \cite[Corollary~3.6]{GS}
A monoid variety is FB whenever it satisfies one of the following:

(i) $\{xtx \approx xtx^2, xy^2 tx \approx xy^2xtx\}$;

(ii) $\{xtx \approx x^2tx, xty^2 x \approx xtxy^2x\}$.

\end{fact}

\begin{lemma}\label{L: not E} Let $\vv$ be a monoid variety that satisfies  $xtx \approx x^2tx \approx xtx^2$ and contains neither  $\mathbb E^1\{\sigma_2\}$ nor $\overline{\mathbb A^1}$.
Then  $\vv$ is FB.
\end{lemma}

\begin{proof}
Consider two cases.

{\bf Case 1:} The set $x^+tx^+$ is not stable with respect to $\vv$.

If $xy$ is not an isoterm for $\vv$ then $\vv$ is either commutative or idempotent (see Lemma~2.6 in \cite{GS}, for instance) and consequently,
is FB \cite{Head-68, Wismath-86}.
If $xy$ is an isoterm for $\vv$ then $\vv \models xtx \approx x^2t$ or  $\vv \models xtx \approx tx^2$. Consequently, $\vv$ satisfies either\\ $xy^2tx \approx xy^2xtx$ or $xty^2x \approx xtxy^2x$ and is  FB by Fact~\ref{F: FB}.

{\bf Case 2:}  The set $x^+tx^+$ is  stable with respect to $\vv$.

 Since $\vv$ does not contain $\mathbb E^1\{\sigma_2\}$, the $\beta$-class  $[at b^2 a]_{\beta}$
is not stable with respect to $\vv$ by Example~\ref{E: E2}.
This means that $\vv \models {\bf u}  \approx {\bf v}$ such that ${\bf u} \in  [at b^2 a]_{\beta}$  but ${\bf v} \not\in  [at b^2 a]_{\beta}$.
Since $\vv \models (yx)^2 \approx (yx)^3$, we may assume that \[{\bf u} \in \{xty^2x, xtyxy, xt (yx)^2, xt (yx)^2y\};\]
\[{\bf v} \in \{xtxy^2, xtxy^2x, xt (xy)^2, xt (xy)^2x\}.\]

Since $\vv$ does not contain  $\overline{\mathbb A^1}$, the dual of Lemma~4.5 in \cite{GS} implies that $\vv \models xty^2x  \approx xt (yx)^2$.
Hence we may assume that \[{\bf u} \in \{xty^2x, xtyxy \}.\]
Since $xty^2x  \approx xt (yx)^2$ implies $xy^2x  \approx x(yx)^2$, we may assume that
\[{\bf v} \in \{xtxy^2, xtxy^2x, xt (xy)^2\}.\]
Then ${\bf u}  \approx {\bf v}$ together with $xtx \approx x^2tx \approx xtx^2$
implies one of the following identities:

{\bf Subcase 2.1:}  $xty^2x \approx xt xy^2x$ or $xty^2x \approx xt (xy)^2x$.

Notice that  $xty^2x \approx xt (xy)^2x$ together with $xy^2x  \approx x(yx)^2$ implies  $xty^2x \approx xt xy^2x$.
Therefore,  $\vv$ is FB by Fact~\ref{F: FB}(ii).

{\bf Subcase 2.2:}  $xtyxy \approx xt xy^2$ or $xtyxy \approx xt (xy)^2$.

Each of these identities together with  $xy^2x  \approx x(yx)^2$    implies $xt (yx)^2 \approx x txy^2x$.

Therefore,  $\vv \models xty^2x \approx xt (yx)^2 \approx xtxy^2x$ and is FB by Fact~\ref{F: FB}(ii).
\end{proof}

\begin{theorem} \label{T: main} ${\mathbb A^1} \vee  {\mathbb E^1}\{\sigma_2\}$ and  $\overline{\mathbb A^1} \vee \overline{ \mathbb E^1\{\sigma_2\}}$ are new limit varieties of monoids.

\end{theorem}

\begin{proof} The variety  ${\mathbb A^1} \vee  {\mathbb E^1}\{\sigma_2\}$ is NFB by Corollary~\ref{C: AE}.

According to \cite[Theorem~4.3(iii)]{Sapir-20+} we have  $\mathbb A^1 = \mathbb M_{\tau_1 \wedge \gamma} ([ab^2t a]_{\tau_1 \wedge \gamma})$ and $\overline{\mathbb A^1} = \mathbb M_{\tau_1 \wedge \gamma}([atb^2 a]_{\tau_1 \wedge \gamma})$. Therefore, the following identity holds on $\mathbb A^1$ but fails on $\overline{\mathbb A^1}$:
\begin{equation}\label{xtysxy}
xtysxy \approx xtysyx.
\end{equation}
Since  ${\mathbb E^1}\{\sigma_2\}$ satisfies  \eqref{xtysxy} by the definition, the variety  ${\mathbb A^1} \vee  {\mathbb E^1}\{\sigma_2\}$
satisfies  \eqref{xtysxy}.

Let $\vv$ be a proper subvariety of   ${\mathbb A^1} \vee  {\mathbb E^1}\{\sigma_2\}$.
If $\vv$ does not contain   ${\mathbb E^1}\{\sigma_2\}$ then $\vv$ is FB by  Lemma~\ref{L: not E}.
So, we may assume that $\vv$ does not contain $\mathbb A^1$.
Then 
\[\vv \models xy^2tx \stackrel{ \cite[Lemma~4.5]{GS}} {\approx}  xyxytx \stackrel{ \eqref{xtysxy}}{\approx} xyyxtx.\]
Hence $\vv$ is FB by Fact~\ref{F: FB}(i).

 The variety $\overline{\mathbb A^1} \vee \overline{ \mathbb E^1\{\sigma_2\}}$ is also  limit  by dual arguments.
\end{proof}


\section{Syntactic monoids $M_{synt} (W)$ (semigroups $S_{synt} (W)$) and monoids of the form $M_\tau(W)$ (semigroups of the form $S_\tau(W)$)}  \label{sec: MW}

 Recall that given a set of words (language) $W \subseteq \mathfrak A^+$,  the syntactic congruence or Myhill  congruence $\sim_W$  on the free monoid  $\mathfrak A^*$ (resp. free semigroup  $\mathfrak A^+$) is defined by
${\bf u} \sim_W {\bf v}$ if and only if for any ${\bf p}, {\bf s} \in \mathfrak A^*$ we have ${\bf p u s} \in W  \Leftrightarrow {\bf p v s} \in W$.  It is well-known and can be easily verified that the syntactic congruence $\sim_W$ is the largest congruence on $\mathfrak A^*$ (resp. $\mathfrak A^+$) for which $W$ is a union of congruence classes. 
The quotient $\mathfrak A^*/ \sim_W$  ($\mathfrak A^+/ \sim_W$)  is called the {\em syntactic monoid} (resp. {\em syntactic semigroup}) of $W$ and denoted by  $M_{synt} (W)$ ( resp. $S_{synt} (W)$) 
 (see \cite{JA}, for instance).

Let $\tau$ be  a congruence  on the free semigroup $\mathfrak A^+$  and   $W \subseteq \mathfrak A^+$ be a union of $\tau$-classes.
We extend $\tau$ to the free monoid $\mathfrak A^*$ by adding $\{(1,1)\}$ to it and define  $S_\tau(W) = M_\tau(W) \setminus \{1\}$.
The following proposition is similar to Proposition~\ref{P: MW} above.

\begin{prop}  \label{P: SW}  Let $\tau$ be a congruence on  the free semigroup $\mathfrak A^+$ and
 $W \subseteq \mathfrak A^+$ be a set of words  which  is a union of  $\tau$-classes.  Let  ${\mathtt W}  \subseteq \mathfrak A^+/ \tau$ denote the set of all $\tau$-classes contained in $W$. Then for every semigroup variety $\vv$  the following are equivalent:

(i)  $\vv$ contains $S_\tau(W) = S_\tau({\mathtt W})$;

(ii)  every word in $W^\le$ is a $\tau$-term for $\vv$;

(iii)  every $\tau$-class in ${\mathtt W}^{\le_\tau}$ is stable with respect to $\vv$.

\end{prop}

\begin{proof}
The equivalence of (i) and (ii) follows Lemma~7.1 in \cite{Sapir-18} and its proof.
The equivalence of  (ii) and (iii) follows from Lemma~2.1 in \cite{Sapir-20+}.
\end{proof}

\begin{obs} \label{O: W2}
(i)  For any congruence $\tau$ on $\mathfrak A^*$  (resp. on $\mathfrak A^+$) and for any
set of words $W \subseteq \mathfrak A^+$ such that $W$ is a union of $\tau$-classes, the syntactic
monoid  $M_{synt} (W)$ (resp. 
semigroup  $S_{synt} (W)$) is a homomorphic image of  $M_\tau(W)$  (resp. $S_\tau(W)$).

(ii) \cite{Jackson-01} The monoids $M(\{{\bf w}\})$ (resp. semigroups $S(\{{\bf w}\})$) and  $M_{synt} (\{{\bf w}\})$ (resp. $S_{synt}(\{{\bf w}\})$) are isomorphic.
\end{obs}

\begin{proof} (i) Since the syntactic congruence $\sim_W$ is larger than $\tau$, the  monoid  $M_{synt} (W) = \mathfrak A^*/ \sim_W$ (resp. semigroup $S_{synt} (W) = \mathfrak A^+/ \sim_W$)
 is a homomorphic image of   $\mathfrak A^*/ \tau$ (resp.  $\mathfrak A^+/ \tau$). The rest follows from the fact that  if  $\mathfrak A^* \setminus W^\le$ is not empty
 then it forms a single $\sim_W$-class.

(ii) If $W$ consists of a single word ${\bf w} \in \mathfrak A^+$ then the syntactic congruence $\sim_W$ is diagonal on $W^\le$. 
\end{proof}

Given a set of words $W$ we
 use $\mathbb M_{synt}(W)$ (resp. $\mathbb S_{synt}(W)$)    to denote the monoid (resp. semigroup)  variety generated by  $M_{synt}(W)$ (resp. $S_{synt}(W)$).
Let $\mathbb S_\tau(W)$ denote the semigroup variety generated by $S_\tau(W)$.

\begin{theorem}  \label{T: SM}  Let $\tau$ be a congruence on the free semigroup  $\mathfrak A^+$.  Let $W \subseteq \mathfrak A^+$ be a set of words such that:

(i)  $W$ forms a single
 $\tau$-class, that is, $W =  [{\bf w}]_\tau$ for some ${\bf w} \in \mathfrak A^+$;

(ii)  if  $W$ is stable with respect to a monoid (resp. semigroup) variety $\vv$,  then every word in $W^\le$ is $\tau$-term for $\vv$.

Then  $\mathbb M_\tau (W) = \mathbb M_{synt} ( W)$ (resp. $\mathbb S_\tau (W) = \mathbb S_{synt} ( W)$).

\end{theorem}

\begin{proof}
 If $M_{synt} (W)$ (resp. $S_{synt} (W)$)  satisfies an identity ${\bf u} \approx {\bf v}$ then  ${\bf u} \sim_W {\bf v}$.
Hence  ${\bf v} \in W$  whenever   ${\bf u} \in W$. Since $W$ is a $\tau$-class we have  ${\bf u} \tau {\bf v}$ whenever   ${\bf u} \in W$. Consequently, every word in $W$ is a $\tau$-term for $M_{synt} (W)$ (resp. $S_{synt} (W)$).
This implies that $W$ is stable with respect to the monoid variety $\mathbb M_{synt} (W)$ (resp. semigroup variety $\mathbb S_{synt} (W)$).  Then  every word in $W^\le$ is a $\tau$-term for  $\mathbb M_{synt} (W)$ (resp.  $\mathbb S_{synt} (W)$) by our assumption.
Using Proposition~\ref{P: MW}   (resp. Proposition~\ref{P: SW})  we obtain  $\mathbb M_\tau(W) \subseteq \mathbb M_{synt} (W)$ ($\mathbb S_\tau(W) \subseteq \mathbb S_{synt} (W)$).
 In view of Observation~\ref{O: W2}, we have $\mathbb M_\tau(W) = \mathbb M_{synt} (W)$  ($\mathbb S_\tau(W) = \mathbb S_{synt} (W)$).
\end{proof}

 {Let $\sim_2$ be the 2-testable congruence on ${\mathfrak A}^*$ defined by

$\bullet$
 ${\bf u} \sim_2 {\bf v}$ if and only if  ${\bf u}$ and  ${\bf v}$ begin and end with the same letter(s) and
 share the same set of subwords of length two.}

 Let $A_2$ and $B_2$ be two 0-simple semigroups given by presentations:
\[A_2 = \langle a, b \mid  a^2 = aba = a, b^2 = 0, bab = b \rangle = \{a, b, ab, ba, 0\};\]
\[B_2 = \langle a, b \mid  a^2 = b^2 = 0, aba = a, bab = b \rangle = \{a, b, ab, ba, 0\}.\]

Let $W_A$ be the set of all words in $\{x,y\}^+$ which do not contain $y^2$ and $W_B$ be the set of all words in $\{x,y\}^+$ which contain neither $x^2$ nor $y^2$.

\begin{ex} Let $\mathbb B_2$ and $\mathbb A_2$ denote the  varieties generated by semigroups $B_2$ and $A_2$ respectively. Then

(i) $\mathbb B_2 = \mathbb S_{\sim_2} (W_B) =   \mathbb S_{\sim_2}([xyx]_{\sim_2})  =  \mathbb S_{synt}(x(yx)^+)$;

(ii) $\mathbb A_2 = \mathbb S_{\sim_2} (W_A) = \mathbb S _{synt}([xyx^2]_{\sim_2})$.

\end{ex}

\begin{proof}(i) 
 It is easy to check that every word in $W_B$ is a $\sim_2$-term for $\mathbb B_2$. Hence  $\mathbb B_2$ contains  $\mathbb S_{\sim_2} (W_B)$ by Proposition~\ref{P: SW}. On  
the other hand,  $S_{\sim_2} (W_B)$ contains the following subsemigroup isomorphic to $B_2$:
\[\{a = [xyx]_{\sim_2}, b = [yxy]_{\sim_2}, ab= [xyxy]_{\sim_2}, ba= [yxyx]_{\sim_2}, 0\}.\]
Therefore, $\mathbb B_2 = \mathbb S_{\sim_2} (W_B)$.  It is easy to see that $W_B^\le = W_B \cup \{1\} = ([xyx]_{\sim_2})^\le$. Thus the semigroups  $S_{\sim_2}(W_B)$ and  $S_{\sim_2}([xyx]_{\sim_2})$ coincide by the definition. 

To verify Condition (ii) in Theorem~\ref{T: SM}, suppose that the set  $[xyx]_{\sim_2} = x(yx)^+$
is stable with respect to a semigroup variety $\vv$.   
 Take some ${\bf v} \in W_B$ and suppose that $\vv \models {\bf v} \approx {\bf w}$
 for some ${\bf w}$. If $\bf w$ contains either $x^2$ or $y^2$ then multiplying  the identity ${\bf v} \approx {\bf w}$ by $x$ on the left or (and) on the right
if necessary, we obtain that  $\vv \models {\bf u} \approx {\bf w'}$ such that ${\bf u} \in x(yx)^+$ but ${\bf w'} \not \in x(yx)^+$. Since this contradicts the fact that
the set $x(yx)^+$ is stable with respect to $\vv$, we must assume that $\bf w$ contains neither $x^2$ nor $y^2$. In a similar way, one can argue that $\bf v$ and $\bf w$
must begin and end with the same letters. Hence every word in $W_B$ must be a  ($\sim_2$)-term for $\vv$. Consequently, $\mathbb S_{\sim_2} ([xyx]_{\sim_2}) = \mathbb S_{synt} ([xyx]_{\sim_2})$
by Theorem~\ref{T: SM}.

(ii) \[\mathbb A_2 \stackrel{*}{=} \mathbb S_{\sim_2} (\mathfrak A^+) \supseteq \mathbb S_{\sim_2}(W_A) \stackrel{**}{=} \mathbb S_{\sim_2}([xyx^2]_{\sim_2}) \stackrel{Theorem~\ref{T: SM}}{=} \mathbb S _{synt}([xyx^2]_{\sim_2}).\]
By the result of Trahtman \cite{Trahtman-81, Trahtman-99},  the relation $\sim_2$ is the fully invariant congruence of  $\mathbb A_2$. Hence $S_{\sim_2} (\mathfrak A^+) = \mathfrak A^+/ \sim_2$ is the relatively free semigroup in $\mathbb A_2$. This establishes
the equality $\stackrel{*}{=}$  above. 
It is easy to see that $W_A^\le = W_A \cup \{1\} =  ([xyx^2]_{\sim_2})^\le$. Thus the semigroups  $S_{\sim_2}(W_A)$ and  $S_{\sim_2}([xyx^2]_{\sim_2})$ coincide by the definition. This establishes
the equality $\stackrel{**}{=}$  above. 
Notice that  the following $5$-elements subset of  $S_{\sim_2}(W_A)$ forms a subsemigroup of  $S_{\sim_2}(W_A)$ isomorphic to $A_2$:
\[\{a = [xyx^2]_{\sim_2}, b = [yx^2y]_{\sim_2}, ab= [x^2yxy]_{\sim_2}, ba= [yxyx^2]_{\sim_2}, 0\}.\]
Hence we have  $\mathbb S_{\sim_2}(W_A) \supseteq \mathbb A_2$ and consequently,  $\mathbb A_2 = \mathbb S_{\sim_2}(W_A) $. Similar arguments
as in the proof of Part (i) show that the semigroups  $S_{\sim_2}([xyx^2]_{\sim_2})$ and  $S_{synt}([xyx^2]_{\sim_2})$ are equationally equivalent 
 by Theorem~\ref{T: SM}.
\end{proof}

\begin{ex} \label{E: subs of E}  Let ${\bf v}_1,  {\bf v}_2, {\bf v}_3, \dots$ be the words defined recursively in \eqref{unvn}.
Then for each $n \ge 1$,  \[\mathbb E^1 \{\sigma_{n+1}\} \stackrel {Example~\ref{E: E2}} {=} \mathbb M_{\beta}  ([{\bf v}_{n}]_\beta) = \mathbb M_{synt}  ([{\bf v}_{n}]_\beta).\]

\end{ex}

\begin{proof}
 $\mathbb M_{\beta}  ([{\bf v}_{n}]_\beta) = \mathbb M_{synt}  ([{\bf v}_{n}]_\beta)$ by  Lemma~\ref{L: un} and Theorem~\ref{T: SM}.
\end{proof}

Let  $L_2$ denote the left-zero semigroup of order two:
\[
L_2 = \langle e, f, \mid e^2 =ef = e,  f^2 = fe=f \rangle=\{e,f\}.
\]

Let  $B_0$ be the semigroup given by presentation:
\[B_0 = \langle e, f, c \mid  e^2=e, f^2 =f, ef=fe=0, ec=cf=c  \rangle,\]
The five-element monoid $B_0^1$ was introduced and shown to be FB in  \cite{Edmunds-77}.

\begin{figure}[htb]
\unitlength=1mm
\linethickness{0.4pt}
\begin{center}
\begin{picture}(55,130)
\put(35,5){\circle*{1.33}}
\put(35,15){\circle*{1.33}}
\put(35,25){\circle*{1.33}}
\put(35,35){\circle*{1.33}}
\put(45,45){\circle*{1.33}}
\put(25,45){\circle*{1.33}}
\put(35,55){\circle*{1.33}}
\put(15,55){\circle*{1.33}}

\put(15,65){\circle*{1.33}}

\put(55,65){\circle*{1.33}}

\put(55,75){\circle*{1.33}}

\put(15,120){\line(0,1){15}}

\put(15, 135){\circle*{1.33}}

\put(15, 125){\circle*{1.33}}

 \put(15,35){\circle*{1.33}}

 \put(15,75){\circle*{1.33}}

\put(15,35){\line(0,1){65}}

\put(35,5){\line(0,1){30}}
\put(35,35){\line(-1,1){20}}
\put(35,35){\line(1,1){10}}
\put(25,45){\line(1,1){10}}

\put(45,45){\line(-1,1){10}}

\put(35,15){\line(-1,1){20}}

\put(35,55){\line(0,1){10}}

\put(35,65){\circle*{1.33}}

\put(35,65){\line(-2,1){20}}

\put(55,85){\circle*{1.33}}

\put(35,65){\line(2,1){20}}

\put(55,65){\line(0,2){20}}

\put(15,105){\circle*{.5}}

\put(15,110){\circle*{.5}}

\put(15,115){\circle*{.5}}

\put(15,95){\circle*{1.33}}

\put(15,85){\circle*{1.33}}

\put(35,55){\line(-2,1){20}}

\put(35,55){\line(2,1){20}}

\put(35,2){\makebox(0,0)[cc]{${\mathbb M}(\varnothing)$}}
\put(37,15){\makebox(0,0)[lc]{${\mathbb M}(1)$}}
\put(37,25){\makebox(0,0)[lc]{${\mathbb M}(x)$}}
\put(37,35){\makebox(0,0)[lc]{${\mathbb M}(xy)$}}
\put(24,45){\makebox(0,0)[rc]{${\mathbb M}_{synt}(ta^+)$}}
\put(14,55){\makebox(0,0)[rc]{${\mathbb M}_{synt}([a^2b^2]_\beta)$}}
\put(71,54){\makebox(0,0)[rc]{$\mathbb B_0^1 = {\mathbb M}_{synt}(a^+ta^+)$}}

\put(14,65){\makebox(0,0)[rc]{${\mathbb M}_{\beta}([a^2b^2]_\beta)$}}

\put(14, 125){\makebox(0,0)[rc]{$\mathbb E^1\{\sigma_\infty\} = {\mathbb M}_\beta(W_\infty)$}}

\put(14, 75){\makebox(0,0)[rc]{$\mathbb E^1\{\sigma_1\}$}}

\put(14, 85){\makebox(0,0)[rc]{$\mathbb E^1\{\sigma_2\} = \mathbb M_{synt}([atb^2a]_\beta)$}}

\put(14, 95){\makebox(0,0)[rc]{$\mathbb E^1\{\sigma_3\} = \mathbb M_{synt}([bsatba]_\beta)$}}

\put(14,35){\makebox(0,0)[rc]{$\mathbb L_2^1 = \mathbb M_{synt} (a \{a,b\}^*)$}}

\put(46,45){\makebox(0,0)[lc]{${\mathbb M_{synt}(a^+t)}$}}
\put(18, 69){\makebox(0,0)[lc]{$ \mathbb Q^1 = {\mathbb M}_{synt}(a^+tsa^+)$}}

\put(56,65){\makebox(0,0)[lc]{$ \mathbb A_0^1 = {\mathbb M}_{synt}(a^+b^+)$}}

\put(56,85){\makebox(0,0)[lc]{$ \mathbb A^1 = {\mathbb M}_{synt}(a^+b^+ta^+)$}}

\put(5, 140){\makebox(0,0)[lc]{${\mathbb E^1} = {\mathbb M}_{synt}([abtab]_\beta)$}}
\end{picture}
\end{center}
\caption{Subvariety lattices of $\mathbb E^1$ (cf. Fig.~4 in \cite{Jackson-Lee}) and of $\mathbb A^1$ (cf. Fig. 1 in \cite{Zhang-Luo})}
\label{pic}
\end{figure}

 In contrast with Example~\ref{E: subs of E}, the next example together with Figure 4 in \cite{Jackson-Lee} implies that
for ${\bf v}_0 = b^2a^2$, the three varieties $\mathbb E^1 \{\sigma_1\}$, $\mathbb M_{\beta}  ( [{\bf v}_0]_\beta)$ and
$\mathbb M_{synt}  ( [{\bf v}_0]_\beta)$ are pairwise distinct (see Figure~\ref{pic}).

\begin{ex} \label{E: LZ}
  \[ (i) \hskip.1in  \mathbb L^1_2 \vee  \mathbb M(x) =  \mathbb M_{synt}  ( [a^2b^2]_\beta).\]

 \[(ii) \hskip.1in \mathbb L^1_2 \vee  \mathbb B^1_0  = \mathbb M_{\beta}  ( [a^2b^2]_\beta).\]

\end{ex}

\begin{proof} 
Let $\alpha$ denote the fully invariant congruence of  $\mathbb L^1_2 = \var\{xy \approx xyx\}$. It is well known and easily verified that for any  ${\bf u}, {\bf v} \in \mathfrak A^*$ we have:

${\bf u}\mathrel{\alpha}{\bf v}$ if and only if  $\con({\bf u}) =  \con({\bf v})$ and
$(_{1{\bf u}} x) <_{\bf u} {(_{1{\bf u}}y)} \Leftrightarrow (_{1{\bf v}} y) <_{\bf v} {(_{1{\bf v}}x)}$ for any $x, y \in \con({\bf u})$.

(i)  
Notice that
$[a^2b^2]_{\alpha \wedge \gamma} = [a^2b^2]_\beta$ is the set of all words in $\{a, b\}^*$ which begin with $a$ and where both $a$ and $b$ are multiple. It is easy to see that if this set
 is stable with respect to a monoid variety $\vv$ then every word in
$([a^2b^2]_{\alpha \wedge \gamma})^\le = \{a, b\}^*$ is $(\alpha \wedge \gamma)$-term for $\vv$. Therefore,
 $\mathbb M_{\alpha \wedge \gamma}  ( [a^2b^2]_{\alpha \wedge \gamma}) =  \mathbb M_{synt}  ( [a^2b^2]_{\alpha \wedge \gamma})$  by Theorem~\ref{T: SM}.
Consequently, $\mathbb M_{\alpha  \wedge \gamma} (\{a,b\}^*) = \mathbb M_{synt}  ( [a^2b^2]_\beta)$.

Since  $\alpha$ is the fully invariant congruence of  $\mathbb L_2^1$ and $\gamma$ is the  fully invariant congruence of  $\mathbb M(x)$,
their intersection  $(\alpha \wedge \gamma)$ is the fully invariant congruence of  $\mathbb L^1_2 \vee  \mathbb M(x)$ and $M_{\alpha  \wedge \gamma}(\mathfrak A^*)$ is the relatively free monoid in  $\mathbb L^1_2 \vee  \mathbb M(x)$.  Therefore, we have
\[\mathbb L^1_2 \vee  \mathbb M(x) = \mathbb M_{\alpha  \wedge \gamma}(\mathfrak A^*) \supseteq \mathbb M_{\alpha  \wedge \gamma} (\{a, b \}^*) \supseteq \mathbb L^1_2 \vee  \mathbb M(x),\]
because $L_2^1$ is isomorphic to the submonoid  $\{1, [a^2b^2]_\beta, [b^2a^2]_\beta\}$  of  $\mathbb M_{\alpha  \wedge \gamma} ( \{a,b\}^*)$ and $x$ is an isoterm for   $\mathbb M_{\alpha  \wedge \gamma} ( \{a,b\}^*)$.
Hence $\mathbb L^1_2 \vee  \mathbb M(x) = \mathbb M_{\alpha \wedge \gamma}  (\{a, b\}^*)$. Overall we have $\mathbb L^1_2 \vee  \mathbb M(x) =  \mathbb M_{synt}  ( [a^2b^2]_\beta)$.

(ii) According to Figure 4 in \cite{Jackson-Lee}, the variety $\mathbb L^1_2 \vee  \mathbb B^1_0$ contains $\mathbb L^1_2 \vee  \mathbb M(x)$.
Hence every word in  $[a^2b^2]_{\alpha \wedge \gamma} = [a^2b^2]_\beta$ is an $(\alpha \wedge \gamma)$-term for $\mathbb L^1_2 \vee  \mathbb B^1_0$ by Part (i) and Proposition~\ref{P: MW}. Consequently, every word in  $[a^2b^2]_{\alpha \wedge \gamma} = [a^2b^2]_\beta$ is a $\beta$-term for $\mathbb L^1_2 \vee  \mathbb B^1_0$. If ${\bf u} \not \in [a^2b^2]_\beta$ but is a subword of a word in $[a^2b^2]_\beta$, then one letter in $\bf u$ is simple. So, modulo renaming letters, ${\bf u} = a^m b a^k$ for some $m, k \ge 0$. According to Theorem~4.1(iii) in \cite{Sapir-20+},
we have $\mathbb B^1_0 =  \mathbb M_{(\tau_1 \wedge \gamma)}(a^+ta^+)$. Together with Proposition~\ref{P: MW}, this implies that $\bf u$ is
$(\tau_1 \wedge \gamma)$-term for $\mathbb L^1_2 \vee  \mathbb B^1_0$. Hence $\bf u$ is $(\alpha \wedge \gamma)$-term for $\mathbb L^1_2 \vee  \mathbb B^1_0$. In view of Proposition~\ref{P: MW}, we have $\mathbb L^1_2 \vee  \mathbb B^1_0  \supseteq \mathbb M_{\beta}  ( [a^2b^2]_\beta)$.

According to Figure 4 in \cite{Jackson-Lee}, the variety $\mathbb L^1_2 \vee  \mathbb B^1_0$ has two dual covers $\mathbb L^1_2 \vee  \mathbb M(x)$ and $\mathbb B^1_0$. Since $\mathbb L^1_2 \vee  \mathbb M(x) \models xtx \approx x^2t$ and $\mathbb B^1_0 \models x^2y^2 \approx y^2x^2$,
Proposition~\ref{P: MW} implies that neither $\mathbb L^1_2 \vee  \mathbb M(x)$ nor $\mathbb B^1_0$ contains $\mathbb M_{\beta}  ( [a^2b^2]_\beta)$. Therefore, $\mathbb L^1_2 \vee  \mathbb B^1_0 = \mathbb M_{\beta}  ( [a^2b^2]_\beta)$.
\end{proof}

Figure~\ref{pic}  combines Figure~4 in \cite{Jackson-Lee} and Figure~1 in \cite{Zhang-Luo}. 
Figure~1 in \cite{Sapir-20+} duplicates the lattice of subvarieties of $\mathbb A^1$ from \cite{Zhang-Luo}, where
each variety $\vv$  is labeled by a monoid of the form $M_{(\tau_1 \wedge \gamma)}(W)$ which generates $\vv$.
When $W$ is a single $(\tau_1 \wedge \gamma)$-class, we use Theorem~\ref{T: SM} together with Lemma~\ref{L: gammasub} to replace the  $M_{(\tau_1 \wedge \gamma)}(W)$  generator of $\vv$ 
by the corresponding syntactic monoid $M_{synt}(W)$. If a variety $\vv$ on Figure~\ref{pic} is a join of two varieties then $\vv$ can be generated by two syntactic monoids.
 For example, if $\vv$ is a non-J-trivial subvariety of  $\mathbb E^1 \{\sigma_{1}\}$ then $\vv$ is the join of $\mathbb L_2^1$ with a subvariety of $\mathbb A^1$.
It is easy to see that the three-element monoid $L^1_2$ is isomorphic to $M_{synt} (a \{a,b\}^*)$, where $a \{a,b\}^*$ is the set of all words in
$\{a,b\}^*$ which begin with $a$.

For each $n \ge 1$ we have  $\mathbb E^1 \{\sigma_{n+1}\} =  \mathbb M_{synt}  ([{\bf v}_{n}]_\beta)$ by Example~\ref{E: subs of E}.
Since $\beta$ is the fully invariant congruence of $\mathbb E^1$, every word is $\beta$-term for $\mathbb E^1$. In particular, every word in  $([abtab]_\beta)^\le$ is $\beta$-term for $\mathbb E^1$.
Hence $\mathbb E^1$ contains  $\mathbb M_\beta([abtab]_\beta)$ by Proposition~\ref{P: MW}. Since $\mathbb E^1\{\sigma_\infty\}$ is a unique maximal subvariety of  $\mathbb E^1$ which satisfies
$xytxy \approx xytyx$ \cite{Jackson-Lee}, we have $\mathbb E^1= \mathbb M_\beta([abtab]_\beta)$. Now suppose that
$[abtab]_\beta$ is stable with respect to a monoid variety $\vv$. 
Since every word in $([abtab]_\beta)^\le$ is either almost-block-simple or belongs to $[abtab]_\beta$,
every word in $([abtab]_\beta)^\le$ is $\beta$-term for $\vv$ by Lemma~\ref{L: un}. Consequently,  $\mathbb E^1= \mathbb M_\beta([abtab]_\beta) = \mathbb M_{synt}([abtab]_\beta)$
by Theorem~\ref{T: SM}.

Overall, using Theorem~\ref{T: SM}, for every variety on Figure~\ref{pic} other than $\mathbb E^1\{\sigma_\infty\} \stackrel{\cite{Jackson-Lee}}{=} \mathbb E\{xytxy \approx xytyx\}$,
one can readily identify one or two syntactic monoids which generate it.
In view of Proposition~5.11 in \cite{Jackson-Lee}, $\mathbb E^1\{\sigma_\infty\}$ is
the only non-finitely generated variety on Figure~\ref{pic}. It is easy to verify that $\mathbb E^1\{\sigma_\infty\} = \mathbb M_\beta(W_\infty)$, where
$W_\infty$ is the set of all almost block-simple words.


\section{Syntactic monoids generating  limit varieties} \label{sec: 14}

Let $\zeta$ be the fully invariant  congruence of $\var \{xtxs \approx xtxsx\}$.  Given $\bf w \in  \mathfrak A^\ast$ we use $\ini_2({\bf w})$ to denote the word obtained by retaining the first two occurrences of each letter in $\bf w$.
 It is easy to see that for every ${\bf u}, {\bf v} \in \mathfrak A^\ast$ we have:
\[{\bf u} \zeta {\bf v} \Leftrightarrow  \ini_2({\bf u}) = \ini_2({\bf v}).\]
Let ${\overline{\zeta}}$ denote the congruence dual to $\zeta$.

The goal of this section is to show that the ten limit varieties mentioned in the introduction can be assembled  from a single `lego set',
which contains ten types of pieces: six words $\{abtbsa, atbsba, atbasb, atb^2a, ab^2ta, ab\}$ and four congruences $\{\tau_1, \gamma, \beta, \zeta\}$. Notice that each of the four congruences $\{\tau_1, \gamma, \beta, \zeta\}$ is defined by a simple formula. We collect the formulas for the syntactic  monoids which generate these varieties in Table~\ref{classes}.

The first two rows of Table~\ref{classes} contain syntactic monoids which generate limit varieties $\mathbb L$ and $\mathbb M$  from \cite{Jackson-05}.
As we mentioned in the introduction,   $\mathbb L$ is generated by $M(\{abtbsa, atbsba\})$  and $\mathbb M$ is generated by $M(\{atbasb\})$.
Observation \ref{O: W2} implies that  $M(\{atbasb\})$ is isomorphic to $M_{synt}(\{atbasb\})$. In view of Lemma~5.1 in \cite{Jackson-Sapir}, we have
$\mathbb L = \mathbb M_{synt}(\{abtbsa\}) \vee \mathbb M_{synt}(\{atbsba\})$.

\begin{fact}\cite[Lemma~6.3]{Sapir-20+} \label{F: two}
Suppose that every letter occurs at most twice in  $\bf u$. If $[{\bf u}]_{\tau_1 \wedge \zeta}$ is stable with respect to a monoid variety
$\vv$ then every word in $([{\bf u}]_{\tau_1 \wedge \zeta})^\le$ is  ($\tau_1 \wedge \zeta$)-term for $\vv$.

\end{fact}

The third row contains  syntactic monoid which generates the limit variety $\mathbb J$ from \cite{Gusev-SF}. According to Theorem~7.2 in \cite{Sapir-20+}, we have
$\mathbb J  = \mathbb M_{\tau_1 \wedge \zeta} ([atbasb]_{\tau_1 \wedge \zeta})$. Using Theorem~\ref{T: SM} together with Fact~\ref{F: two}, we obtain
\[\mathbb J  = \mathbb M_{synt} ([atbasb]_{\tau_1 \wedge \zeta}) = \mathbb M_{synt}  (atba^+sb^+).\]
Dually, we have
\[\overline{\mathbb J} =  \mathbb M_{\tau_1 \wedge \zeta} ([atbasb]_{\tau_1 \wedge \overline{\zeta}}) = \mathbb M_{synt}  (a^+tb^+asb).\]

The limit variety $\mathbb K$ in the fourth row of Table~\ref{classes} was introduced in \cite{GS} as the variety generated by $M_{\tau_1 \wedge \zeta} ([atb^2a]_{\tau_1 \wedge \zeta})$.
Using Theorem~\ref{T: SM} together with Fact~\ref{F: two}, we obtain
\[\mathbb K  = \mathbb M_{synt} ([atb^2a]_{\tau_1 \wedge \zeta}) = \mathbb M_{synt}  (atbb^+a^+).\] 
Dually, we have
\[\overline{\mathbb K} =  \mathbb M_{\tau_1 \wedge \overline{\zeta}} ([a b^2t a]_{\tau_1 \wedge \overline{\zeta}}) = \mathbb M_{synt}  (a^+b^+bt a).\]
Using Fact~\ref{F: two},  one can verify that the variety
$\mathbb K$ can be also generated  by \\ $M_{\tau_1 \wedge \zeta} ([atbsba]_{\tau_1 \wedge \zeta}) = M_{\tau_1 \wedge \zeta} (atbsb^+a^+) $. Hence, we also have 
\[\mathbb K  = \mathbb M_{synt} (atbsb^+a^+),   \overline{\mathbb K} =  \mathbb M_{synt}(a^+b^+tbsa).\]

The fifth row contains two syntactic monoids which generate the limit variety $\overline{\mathbb A }^1 \vee \mathbb A^1$ from \cite{Zhang-Luo}.
Indeed, according to Theorem~4.3 in \cite{Sapir-20+}, the variety  $\mathbb A^1$ is generated by  $M_{\tau_1 \wedge \gamma}([ab^2t a]_{\tau_1 \wedge \gamma}) =M_{\tau_1 \wedge \gamma}(a^+b^+t a^+)$
and $\overline{\mathbb A }^1$ is generated by  $M_{\tau_1 \wedge \gamma}([atb^2 a]_{\tau_1 \wedge \gamma}) =M_{\tau_1 \wedge \gamma}(a^+t b^+a^+)$.
Using Theorem~\ref{T: SM} together with Lemma~\ref{L: gammasub}, we obtain 
\[\overline{\mathbb A^1} \vee \mathbb A ^1 = \mathbb M_{synt} (a^+t bb^+a^+) \vee \mathbb M_{synt} (a^+bb^+t a^+).\]

The sixth row in Tables~\ref{classes} contains two syntactic monoids which generate the limit variety
$\overline{\mathbb A^1} \vee \overline{ \mathbb E^1\{\sigma_2\}}$ from Theorem~\ref{T: main} above.
In view of the previous paragraph and the dual of Example~\ref{E: subs of E}, we have
\[\overline{\mathbb A^1} \vee \overline{ \mathbb E^1\{\sigma_2\}} =   \mathbb M_{synt} (a^+t bb^+a^+) \vee   \mathbb M_{synt}  ([ab^2 t a]_{\overline{\beta}}).\]
Dually, we have
\[ {\mathbb A^1} \vee  {\mathbb E^1}\{\sigma_2\}    =   \mathbb M_{synt} (a^+bb^+ta^+) \vee   \mathbb M_{synt}  ([a tb^2a]_\beta).\]

Finally, the seventh row  contains three syntactic monoids which generate the limit variety 
${\mathbb E}^1(\sigma_2) \vee \overline{{\mathbb E}^1(\sigma_2)} \vee \mathbb A_0^1$ from \cite{GS-23}.
According to Sect.~7 in \cite{Sapir-18}, the monoid $A_0^1$ is isomorphic to $M_{\tau_1}([ab]_{\tau_1})  =  M_{\tau_1} (a^+b^+).$
Using Theorem~\ref{T: SM} together with Lemma~\ref{L: gammasub}, we obtain $\mathbb A_0^1 = \mathbb M_{synt} (a^+b^+)$.  Example~\ref{E: subs of E} and its dual imply that
\[{\mathbb E}^1(\sigma_2) \vee \overline{{\mathbb E}^1(\sigma_2)} \vee \mathbb A_0^1 =  \mathbb M_{synt}([at b^2 a]_{\beta}) \vee  \mathbb M_{synt}([a b^2t  a]_{\overline{\beta}}) \vee
\mathbb M_{synt}(a^+b^+).\]

In contrast with the ten limit varieties whose generating syntactic monoids up to duality are listed in the first seven rows of Table~\ref{classes}, it seems that monoids of the form $M_\tau(W)$ are not useful for describing the (finitely generated) limit varieties $\mathbb J_1$, $\mathbb J_2$, $\overline{\mathbb J_1}$ and $\overline{\mathbb J_2}$  found in \cite{GLZ}.
For instance, Lemma~5.3  in \cite{GLZ} implies that $\mathbb J_1 =
\mathbb M_\delta( [asbtabzzs]_\delta)$, for some congruence $\delta$ such that $[asbtabzzs]_\delta = [asbtabzzs]_{\zeta} \cup [asbtabzsz]_{\zeta}$. The formula for the congruence $\delta$ can be extracted from the proof of Proposition~5.4  in \cite{GLZ}
and is very technical.

\begin{table}[tbh]
\begin{center}
\small
\begin{tabular}{|l|l|l|}
\hline limit variety of monoids generated by&   found by   \\
\hline
\protect\rule{0pt}{10pt}    $M_{synt} (\{abtbsa\}) \times M_{synt} ( \{atbsba\})$  & Jackson in \cite{Jackson-05}\\
\hline
\protect\rule{0pt}{10pt}    $M_{synt} (\{atbasb\})$  &  Jackson  in \cite{Jackson-05} \\
\hline
\protect\rule{0pt}{10pt}   $M_{synt}  ([atbasb]_{\tau_1 \wedge \zeta})$  & Gusev in \cite{Gusev-SF} \\
\hline
\protect\rule{0pt}{10pt}    $M_{synt}  ([atb^2a]_{\tau_1 \wedge \zeta})$ or by $M_{synt}  ([atbsba]_{\tau_1 \wedge \zeta})    $   & Gusev-Sapir in \cite{GS} \\
\hline
\protect\rule{0pt}{10pt}   $M_{synt}  ([atb^2a]_{\tau_1 \wedge \gamma}) \times  M_{synt}  ([ab^2ta]_{\tau_1 \wedge \gamma})$
 & Zhang-Luo in \cite{Zhang-Luo} \\
\hline
\protect\rule{0pt}{10pt}  $M_{synt}  ([at b^2a]_{\tau_1 \wedge \gamma}) \times M_{synt}  ([ab^2 t a]_{\overline{\beta}})$   & Theorem~\ref{T: main}  \\
\hline
\protect\rule{0pt}{10pt}  $M_{synt}  ([at b^2 a]_\beta) \times M_{synt}  ([ab^2 t a]_{\overline{\beta}})\times M_{synt}( [a b]_{\tau_1})$   &  Gusev-Sapir in \cite{GS-23} \\
\hline

\protect\rule{0pt}{10pt} Contains $M_{synt}( [atbsba]_\zeta ) $ & Gusev-Li-Zhang in \cite{GLZ} \\
\hline
\protect\rule{0pt}{10pt}
Contains $M_{synt}( [atb^2a]_\zeta ) $   &  Gusev-Li-Zhang in \cite{GLZ} \\
\hline
\end{tabular}
\caption{14  limit varieties of aperiodic monoids
\protect\rule{0pt}
{11pt}}
\label{classes}
\end{center}
\end{table}


\section*{Acknowledgments}
The author is grateful to Sergey Gusev for many discussions and for reading  several versions of this paper with lots of valuable comments.

\end{document}